\documentclass[11pt]{amsart}

\usepackage{amsmath, fullpage, amssymb, latexsym}
\usepackage{epsfig}
\usepackage{amsfonts,amsthm,mathrsfs,hyperref}
\usepackage{mathtools}
\usepackage{makecell}
\usepackage{todonotes}
\usepackage{xcolor}
\usepackage{tikz}
\usepackage{tikz-cd}
\usepackage{pgf}
\usepackage{fancyvrb}   
\usepackage{booktabs}
\usepackage{diagbox}
\usepackage{listings}

\definecolor{forestgreen(traditional)}{rgb}{0.0, 0.27, 0.13}
\definecolor{forestgreen(web)}{rgb}{0.13, 0.55, 0.13}
\definecolor{airforceblue}{rgb}{0.36, 0.54, 0.66}
\definecolor{bleudefrance}{rgb}{0.19, 0.55, 0.91}
\definecolor{darkorchid}{rgb}{0.6, 0.2, 0.8}
\definecolor{darkorange}{rgb}{1.0, 0.55, 0.0}
\definecolor{darkspringgreen}{rgb}{0.09, 0.45, 0.27}

\newtheorem{theorem}{Theorem}[section]
\newtheorem{corollary}[theorem]{Corollary}
\newtheorem{lemma}[theorem]{Lemma}
\newtheorem{proposition}[theorem]{Proposition}

\newtheorem{notation and remarks}[theorem]{Definitions and Notations}
\newtheorem{notation and remark}[theorem]{Notation and Remark}
\theoremstyle{definition}

\newtheorem{example}[theorem]{Example}

\theoremstyle{remark}
\newtheorem{remark}[theorem]{Remark}

\def\kk{{\Bbbk}}
\DeclareMathOperator{\rank}{rank}
\def\ds{\displaystyle}

\newcommand \ti[1]{\textit{#1}}

\begin{document}

\title[A family of explicit Waring decompositions of a polynomial]
{A family of explicit Waring decompositions of a polynomial}

\author[K. Han]{Kangjin Han}
\address{Kangjin Han, School of Undergraduate Studies,
Daegu-Gyeongbuk Institute of Science \& Technology (DGIST),
Daegu 42988,
Republic of Korea}
\email{kjhan@dgist.ac.kr}

\author[H. Moon]{Hyunsuk Moon}
\address{Hyunsuk Moon, School of Mathematics, Korea Institute for Advanced Study (KIAS), Seoul 02455, Republic of Korea}
\email{hsmoon87@kias.re.kr, mhs@kaist.ac.kr}

\subjclass[2010]{14P99, 12D05, 13P10, 14A25, 15A21, 14N15}
\keywords{Waring rank, Waring decomposition, Monomials, Symmetric tensor, Complexity}

\begin{abstract}
In this paper we settle some polynomial identity which provides a family of explicit Waring decompositions of any monomial $X_0^{a_0}X_1^{a_1}\cdots X_n^{a_n}$ over a field $\Bbbk$. This gives an upper bound for the Waring rank of a given monomial and naturally leads to an explicit Waring decomposition of any homogeneous form and, eventually, of any polynomial via (de)homogenization. Note that such decomposition is very useful in many applications dealing with polynomial computations, symmetric tensor problems and so on. We discuss some computational aspect of our result as comparing with other known methods and also present a computer implementation for potential use in the end.
\end{abstract}

\maketitle

\section{Introduction}

Throughout the paper, let $\kk$ be an infinite field and $R=\kk[X_0,\ldots ,X_n]=\bigoplus_{d\ge0} R_d$ be the ring of polynomials over $\kk$ and $R_d$ be the $\kk$-vector space of homogeneous polynomials (or forms) of degree $d$.
For a given $F \in R_d$, a \textit{Waring decomposition of $F$ over $\kk$} is defined as a sum 
\begin{equation}\label{waring_decomp}
F=\sum_{i=1}^{r} \lambda_i L_i^d~,
\end{equation}
where $\lambda_i \in \kk$ and $L_i$ is a linear form over $\kk$. The smallest number $r$ for which such a decomposition exists is called \textit{Waring rank of $F$ over $\kk$} and we denote it by $\rank_\kk(F)$.

Earlier studies of Waring decomposition and Waring rank, initiated by works of Sylvester and others, go back to the 19th century (see \cite{IK} for a historical background). But, despite their long history, the Waring ranks for \textit{general} forms over the complex numbers, a long-standing conjecture in this field, were determined only in the 1990s by \cite{AH} and the complex Waring rank of monomials, a specific type of polynomial, has been found in recently \cite{BBT, CCG}. Generally, it turns out that determination of the rank of a form is a very difficult task except some known cases (see e.g. \cite{HL} and references therein).

Over the \textit{real} numbers the Waring rank is even more difficult to compute in general. For some basic cases, the real Waring rank of a monomial  $X_0^{a_0}X_1^{a_1}\cdots X_n^{a_n}$ with $a_i>0$ is known; it is the degree when $n=1$ \cite{BCG} and $\frac{1}{2}\prod_{i=0}^n(a_i+1)$ when $\min(a_i)=1$ \cite{CKOV}. In \cite{CKOV}, they also provide an upper bound for the real rank of any monomial $\frac{1}{2a_j}\prod_{i=0}^n(a_i+a_j)$ where $a_j$ is $\min(a_i)$, but it is not tight in general. The state of art result has been obtained recently by authors in \cite{HM} as follows:

\begin{theorem}[\cite{HM}]\label{main_thm}
For a monomial $X_0^{a_0}X_1^{a_1}\ldots X_n^{a_n}$ with each $a_i>0$, it holds that 
\begin{equation}\label{main_bd}
\rank_\mathbb{R}(X_0^{a_0}X_1^{a_1}\ldots X_n^{a_n})\le\frac{1}{2}\bigg\{\prod_{i=0}^n(a_i+1)-\prod_{i=0}^n(a_i-1)\bigg\}~.
\end{equation}
Further, the same result is true for $\rank_\mathbb{Q}(X_0^{a_0}X_1^{a_1}\ldots X_n^{a_n})$.
\end{theorem}

In this article we settle some identity which concerns an explicit Waring decompositions of any monomial. This can be used to recover the bound (\ref{main_bd}) in a direct way. In principle, via `apolarity' one could find an Waring decomposition of a given monomial $M$ over $\mathbb{R}$ or $\mathbb{Q}$ using the sub-ideal of the apolar ideal $M^\perp$ which appeared in \cite{HM}. But, in general this involves such a huge amount of linear algebra computation of a large size system to determine whole coefficients of the decomposition that it is not easy to get the result actually in many cases. 

Here we prove some Waring-type identity which is parametrized by most numbers in the (infinite) base field $\kk$ and has an interesting combinatorial nature (see Theorem \ref{coeff} and Corollary \ref{reduced_decomp}). As a result, we can have a more direct formula for a Waring decomposition of $M$ without relying on a massive linear algebra calculation.

We discuss some consequence of the identities in Section \ref{sect_main} for finding an explicit Waring decomposition of not only a monomial, but also of \textit{any} homogeneous form (to an arbitrary polynomial it can be easily applied via the process of (de)homogenization, too). Especially, in Section \ref{sect_conseq} we consider its computational aspect; it turns out that our method is asymptotically better in the number of summands than the method in \cite{BBDKV11} using a previously known decomposition (see Remark \ref{F&K} for further discussion). We also present an example of the identity in Example \ref{432case} and a software implementation using a symbolic computer algebra system \textsc{Macaulay2} \cite{M2} as well in Section \ref{sect_M2code}. 

Finally, we'd like to mention that almost everything in the paper also does work over a finite field provided that the characteristic is relatively large. But, for brevity we here focus only on the case of an infinite field $\kk$.

\section{Main Result}\label{sect_main}

\begin{notation and remarks}
First of all we define some notions and set notations on them.

\begin{enumerate}
\item Let $\mathbf{a}\in\mathbb{Z}_{\geq 0}^{n+1}$ be a sequence of $n+1$ nonnegative integers (note that such an $\mathbf{a}$ naturally corresponds to a monomial $X_0^{a_0}X_1^{a_1}\cdots X_n^{a_n}\in \Bbbk[X_0,\ldots,X_n]$). $|\mathbf{a}|$ means its sum $\sum_0^n a_i$ and ${|\mathbf{a}|\choose \mathbf{a}}$ denotes the multinomial coefficient $\binom{|\mathbf{a}|}{a_0,a_1,\dots,a_n}$.
\item For a given $\mathbf{a}\in\mathbb{Z}_{\geq 0}^{n+1}$, we set $Z_{\mathbf{a}}:=\{i~|~a_i=0\}$, the set of indices of zeros, $E_{\mathbf{a}}:=\{i~|~a_i \mathrm{~is~even}\}$, the set of even indices and  $m_i=\lfloor\frac{a_i-1}{2}\rfloor$ for $i=0,\ldots,n$.
\item Fix $\mathbf{a}\in\mathbb{Z}_{\geq 0}^{n+1}$. For any set $A$ with $Z_\mathbf{a} \subset A\subset E_\mathbf{a}$, we consider a set of ordered multiples $K_A:=\{(k_i)_{i\not\in A}~|~k_i\in \mathbb{Z}, 0\leq k_i\leq m_i \mathrm{~for~} i\not\in A\}$ and $S_A:=\{(s_i)_{i\not\in A}~|~ s_i\in\{0,1\}\}$ which is the set of all ordered multiples for signs. As a reduction of each set, we also define two specific subsets
\begin{align*}
\overline{K_A}&:=\{(k_i)_{i\not\in A}~|~k_i\in \mathbb{Z}, 0\leq k_i\leq m_i \mathrm{~for~}  i\not\in A \mathrm{~and~} \min\{k_i\}_{i\not\in A}=0\}\subset K_A~,\\
\overline{S_A}&:=\{(s_i)_{i\not\in A}~|~ s_i\in\{0,1\}, s_{\min\{i:i\not\in A\}}=0\}\subset S_A~.
\end{align*}
\item For a given triple $(A,\mathbf{k},\mathbf{s})$ where $A$ is a set such that $Z_\mathbf{a}\subset A\subset E_\mathbf{a}$, $\mathbf{k}\in K_A$ and $\mathbf{s}\in {S}_A$, we set $\ds\ell_{A,\mathbf{k},\mathbf{s}}:=\sum_{i\not\in A} (-1)^{s_i}t^{k_i}X_i$, a linear form in $\kk[X_0,X_1,\dots,X_n]$.
\item Finally, we need to define the following combinatorial function depending on the value of $a_i$;
$$F_i(y):=
\begin{cases}
\ds\prod_{j=1}^{m_i}(y-t^{a_i-2j}) \quad\emph{,~for~} m_i>0\\
1 \quad\quad \quad\quad\quad \quad\quad\emph{,~if~} m_i=-1\emph{~or~}0\end{cases}~. $$
Then $F_i(y)$ is a polynomial function in $\kk[y]$  and has degree $m_i$ unless $m_i=-1$. For a monomial $M$ and a polynomial $f$ in $\Bbbk[y]$, we denotes the coefficient of $M$ in the polynomial $f$ by $c(M,f)$.
\end{enumerate}

\end{notation and remarks}

%

For the proof of Theorem \ref{coeff}, we also prove a lemma on the sum of signs.

\begin{lemma}\label{lemI}
For a sequence of integers $J\in \mathbb{Z}^k$,
$$\sum_{I\in\{0,1\}^k} (-1)^{\sum_{i=1}^k J_i\cdot I_i}=\prod_{i=1}^k(1+(-1)^{J_i})=
\begin{cases}
2^k, \emph{~if all the~} J_i \emph{~are even}\\
0, \emph{~if at least one of~} J_i \emph{~is odd}\end{cases}~,$$
where $\{0,1\}^k$ denotes the set of all sequences of $k$ numbers whose entry is $0$ or $1$.
\end{lemma}
\begin{proof}
\begin{align*}
\sum_{I\in\{0,1\}^k} (-1)^{\sum_{i=1}^k J_i\cdot I_i}&=\sum_{I\in\{0,1\}^k} \prod_{i=1}^k (-1)^{J_i\cdot I_i}\\&=\sum_{0\leq I_1\leq 1}\sum_{0\leq I_2\leq 1}\cdots \sum_{0\leq I_k\leq 1}(-1)^{J_1\cdot I_1}(-1)^{J_2\cdot I_2}\cdots (-1)^{J_k\cdot I_k}\\
&=\sum_{0\leq I_1\leq 1}(-1)^{J_1\cdot I_1}\cdots \sum_{0\leq I_k\leq 1}(-1)^{J_k\cdot I_k}\\
&=\prod_{i=1}^k \sum_{0\leq I_k\leq 1} (-1)^{J_i\cdot I_i}=\prod_{i=1}^k (1+(-1)^{J_i})
\end{align*}
\end{proof}

Here is a small example of the equality of Lemma \ref{lemI}. 
\begin{example}
Let $J=(3,4,5)$. Then $k=3$ and the set $\{0,1\}^3$ is
$$\{(0,0,0),(0,0,1),(0,1,0),(0,1,1),(1,0,0),(1,0,1),(1,1,0),(1,1,1)\}.$$
For each $I\in \{0,1\}^3$, $\sum_{i=1}^3 I_i\cdot J_i=I_1J_1+I_2J_2+I_3J_3$ is given as
$$0,5,4,9,3,8,7,~\textrm{and}~12~.$$
Hence, in one side we have 
\begin{align*}
\sum_{I\in \{0,1\}^3} (-1)^{\sum_{i=1}^3 I_i\cdot J_i}&=(-1)^0+(-1)^5+(-1)^4+(-1)^9+(-1)^{3}+(-1)^{8}+(-1)^{7}+(-1)^{12}\\
&=1-1+1-1-1+1-1+1=0
\end{align*}
and in the other side
\begin{align*}
&\prod_{i=1}^k(1+(-1)^{J_i})=(1+(-1)^3)(1+(-1)^4)(1+(-1)^5)=0~.
\end{align*}
\end{example}

Now we prove our main theorem.

\begin{theorem}\label{coeff}
Let $\mathbf{a}\in\mathbb{Z}_{\geq 0}^{n+1}$ be a sequence of $n+1$ nonnegative integers with $d=|\mathbf{a}|$. 
Then, it holds that
\begin{equation}\label{eqeq0}
D_\mathbf{a}\cdot X_0^{a_0}X_1^{a_1}\cdots X_n^{a_n}=\sum_{Z_\mathbf{a}\subset A\subset E_\mathbf{a}}\sum_{(\mathbf{k},\mathbf{s})\in K_A\times \overline{S_A}}C_{A,\mathbf{k},\mathbf{s}}(\ell_{A,\mathbf{k},\mathbf{s}})^d\end{equation}
where
$$D_\mathbf{a}=(-1)^{|Z_\mathbf{a}|}\cdot 2^n \cdot \binom{d}{\mathbf{a}}\cdot\prod_{i\not\in Z_\mathbf{a}} F_i(t^{a_i})$$ and
$$C_{A,\mathbf{k},\mathbf{s}}=(-1)^{|A|}\cdot 2^{|A|}\cdot (-1)^{\sum_{i\not\in A}a_i\cdot s_i}\prod_{i\in A} F_i(1) \prod_{i\not\in A} c(y^{k_i},F_i) $$

\end{theorem}
\begin{proof}

$$(\ell_{A,\mathbf{k},\mathbf{s}})^d=\big(\sum_{i\not\in A} (-1)^{s_i}t^{k_i}X_i\big)^d=\sum_{|\mathbf{b}|=d,A\subset Z_{\mathbf{b}}}\binom{d}{\mathbf{b}}\prod_{i\not\in A} (-1)^{s_i\cdot b_i}(t^{k_i})^{b_i} X_i^{b_i}.$$
Rewrite the right hand side of the equation as the coefficient and monomial with degree $d$ as follows
\begin{align*}
&\sum_{Z_\mathbf{a}\subset A\subset E_\mathbf{a}}\sum_{(\mathbf{k},\mathbf{s})\in K_A\times \overline{S_A}}C_{A,\mathbf{k},\mathbf{s}}\bigg(\sum_{|\mathbf{b}|=d,A\subset Z_{\mathbf{b}}}\binom{d}{\mathbf{b}}\prod_{i\not\in A} (-1)^{s_i\cdot b_i}(t^{k_i})^{b_i} X_i^{b_i}\bigg)\\
&=\sum_{Z_\mathbf{a}\subset A\subset E_\mathbf{a}}\sum_{|\mathbf{b}|=d,A\subset Z_{\mathbf{b}}}\binom{d}{\mathbf{b}}\prod_{i\not\in A}X_i^{b_i}\sum_{(\mathbf{k},\mathbf{s})\in K_A\times \overline{S_A}}C_{A,\mathbf{k},\mathbf{s}}\bigg(\prod_{i\not\in A} (-1)^{s_i\cdot b_i}(t^{k_i})^{b_i}\bigg)\\
&=\sum_{|\mathbf{b}|=d}\binom{d}{\mathbf{b}}\prod_{i=0}^nX_i^{b_i}\sum_{Z_\mathbf{a}\subset A\subset E_\mathbf{a}, A\subset Z_\mathbf{b}}\sum_{(\mathbf{k},\mathbf{s})\in K_A\times \overline{S_A}}C_{A,\mathbf{k},\mathbf{s}}\bigg(\prod_{i\not\in A} (-1)^{s_i\cdot b_i}(t^{k_i})^{b_i}\bigg)\\
\end{align*}

Let $$T_\mathbf{b}=\sum_{Z_\mathbf{a}\subset A\subset E_\mathbf{a}, A\subset Z_\mathbf{b}}\sum_{(\mathbf{k},\mathbf{s})\in K_A\times \overline{S_A}}C_{A,\mathbf{k},\mathbf{s}}\bigg(\prod_{i\not\in A} (-1)^{s_i\cdot b_i}(t^{k_i})^{b_i}\bigg)$$
then 
$$\sum_{Z_\mathbf{a}\subset A\subset E_\mathbf{a}}\sum_{(\mathbf{k},\mathbf{s})\in K_A\times \overline{S_A}}C_{A,\mathbf{k},\mathbf{s}}(\ell_{A,\mathbf{k},\mathbf{s}})^d=\sum_{|\mathbf{b}|=d}\binom{d}{\mathbf{b}}\cdot T_\mathbf{b}\cdot X_0^{b_0}X_1^{b_1}\cdots X_n^{b_n}$$

We want to show that all $T_\mathbf{b}=0$ except $\mathbf{b}\neq\mathbf{a}$ and $\binom{d}{\mathbf{a}}T_\mathbf{a}=D_\mathbf{a}$.\\

\noindent\textbf{Step 1} 

Since all the linear forms $\ell_{A,\mathbf{k},\mathbf{s}}$ do not have variables $X_i$ with $i\in Z_\mathbf{a}$, $T_\mathbf{b}=0$ for any $\mathbf{b}$ with $Z_\mathbf{a}\not\subset Z_\mathbf{b}$. Hence we can assume that $Z_\mathbf{a}\subset Z_\mathbf{b}$. So we can reduce
$$T_\mathbf{b}=\sum_{Z_\mathbf{a}\subset A\subset E_\mathbf{a}\cap Z_\mathbf{b}}\sum_{(\mathbf{k},\mathbf{s})\in K_A\times \overline{S_A}}C_{A,\mathbf{k},\mathbf{s}}\bigg((-1)^{\sum_{i\not\in A} s_i\cdot b_i}\prod_{i\not\in A}(t^{k_i})^{b_i}\bigg)$$

\noindent\textbf{Step 2} 

\begin{align*}
T_\mathbf{b}=\sum_{Z_\mathbf{a}\subset A\subset E_\mathbf{a}\cap Z_\mathbf{b}}\sum_{\mathbf{k}\in K_A}\sum_{\mathbf{s}\in \overline{S_A}}&\Big\{(-1)^{|A|}\cdot 2^{|A|}\cdot(-1)^{\sum_{i\not\in A}a_i\cdot s_i}\cdot \prod_{i\in A}F_i(1)\prod_{i\not\in A} c(y^{k_i},F_i)\\
&\cdot(-1)^{\sum_{i\not\in A}b_i\cdot s_i}\prod_{i\not\in A} (t^{k_i})^{b_i}\Big\}
\end{align*}
\begin{align*}
=\sum_{Z_\mathbf{a}\subset A\subset E_\mathbf{a}\cap Z_\mathbf{b}}(-1)^{|A|}\cdot 2^{|A|}\cdot \prod_{i\in A} F_i(1)\cdot&\bigg\{\sum_{\mathbf{k}\in K_A} \prod_{i\not\in A}c(y^{k_i},F_i)\prod_{i\not\in A} (t^{k_i})^{b_i}\\
&\cdot\sum_{\mathbf{s}\in \overline{S_A}}(-1)^{\sum_{i\not\in A}a_i\cdot s_i}(-1)^{\sum_{i\not\in A}b_i\cdot s_i}\bigg\}
\end{align*}
By Lemma \ref{lemI}, 
$$\sum_{\mathbf{s}\in \overline{S_A}} (-1)^{\sum_{i\not\in A}a_i\cdot s_i}(-1)^{\sum_{i\not\in A}b_i\cdot s_i}=\sum_{\mathbf{s}\in \overline{S_A}} (-1)^{\sum_{i\not\in A}(a_i+b_i)\cdot s_i}=
\begin{cases}
2^{n-|A|}, &a_i+b_i\equiv 0(\mbox{mod~} 2) \mbox{~for all~} \\&i\not\in A, i\neq \min\{i:i\not\in A\}\\
0, &\mbox{otherwise}
\end{cases}
$$
since $s_{\min\{i\not\in A\}}=0$. Let $\delta$ be the map from $\mathbb{Z}^2$ to $\{0,1\}$ such that $\delta(x,y)=\begin{cases} 1, \mbox{~if~} x\equiv y(\mbox{mod} 2)\\ 0, \mbox{~if~} x\not\equiv y(\mbox{mod} 2)\end{cases}$. Then 
$$\sum_{\mathbf{s}\in \overline{S_A}} (-1)^{\sum_{i\not\in A}a_i\cdot s_i}(-1)^{\sum_{i\not\in A}b_i\cdot s_i}=2^{n-|A|}\prod_{i\not\in A, i\neq\min\{i:i\not\in A\}}\delta(a_i,b_i)$$
For $i\in A\subset E_\mathbf{a}\cap Z_\mathbf{a}$, $a_i$ and $b_i$ are both even. Hence $d=\sum_{i=0}^n a_i=\sum_{i=0}^n b_i$ implies $\sum_{i\not\in A} a_i\equiv \sum_{i\not\in A} b_i$. So if $a_i$ and $b_i$ have same parity for all $i\not\in A, i\neq\min\{i:i\not\in A\}$, then $a_{\min\{i:i\not\in A\}}$ and $b_{\min\{i:i\not\in A\}}$ also have same parity. It means that
$$\prod_{i\not\in A, i\neq\min\{i:i\not\in A\}}\delta(a_i,b_i)=\prod_{i\not\in A} \delta(a_i,b_i)$$
Since $2^{|A|}\cdot 2^{n-|A|}=2^n$ do not relate to any summation and $\prod_{i\not\in A}\delta(a_i,b_i)$ is only depend on the choice of $A$, 
\begin{align*}
T_\mathbf{b}=2^n\cdot \sum_{Z_\mathbf{a}\subset A\subset E_{\mathbf{a}}\cap Z_\mathbf{b}} (-1)^{|A|}\prod_{i\in A}F_i(1)\prod_{i\not\in A}\delta(a_i,b_i)\bigg\{\sum_{\mathbf{k}\in K_A}\prod_{i\not\in A} c(y^{k_i},F_i)\prod_{i\not\in A}(t^{k_i})^{b_i}\bigg\}
\end{align*}

\noindent\textbf{Step 3}
Let $\{0,\ldots,n\}\backslash A=\{i_1,i_2,\ldots,i_p\}$ where $p=n+1-|A|$. Then
\begin{align*}
&\sum_{\mathbf{k}\in K_A}\prod_{i\not\in A} c(y^{k_i},F_i)\prod_{i\not\in A}(t^{k_i})^{b_i} \\
&=\sum_{0\leq k_{i_1}\leq m_{i_1}}\sum_{0\leq k_{i_2}\leq m_{i_2}}\cdots\sum_{0\leq k_{i_p}\leq m_{i_p}}\prod_{j=1}^p c(y^{k_{i_j}},F_{i_j})(t^{k_{i_j}})^{b_{i_j}}\\
&=\sum_{0\leq k_{i_1}\leq m_{i_1}}c(y^{k_{i_1}},F_{i_1})(t^{k_{i_1}})^{b_{i_1}}\cdots\sum_{0\leq k_{i_p}\leq m_{i_p}}c(y^{k_{i_p}},F_{i_p})(t^{k_{i_p}})^{b_{i_p}}\\\
&=\prod_{j=1}^p\sum_{0\leq k_{i_j}\leq m_{i_j}} c(y^{k_{i_j}},F_{i_j})(t^{k_{i_j}})^{b_{i_j}}=\prod_{i\not\in A}\sum_{0\leq k_i\leq m_i}c(y^{k_i},F_i)(t^{k_i})^{b_i}\\
&=\prod_{i\not\in A}\sum_{0\leq k_i\leq m_i}c(y^{k_i},F_i)(t^{b_i})^{k_i}
\end{align*}
Since each $F_i$ has degree $m_i$ for $i\not\in A$, $\sum_{0\leq k_i\leq m_i}c(y^{k_i},F_i)(t^{b_i})^{k_i}=F_i(t^{b_i})$. Hence 
$$T_\mathbf{b}=2^n\cdot \sum_{Z_\mathbf{a}\subset A\subset E_{\mathbf{a}}\cap Z_\mathbf{b}} (-1)^{|A|}\prod_{i\in A}F_i(1)\prod_{i\not\in A} F_i(t^{b_i}) \prod_{i\not\in A} \delta(a_i,b_i)$$

\noindent\textbf{Step 4}

Since $F_i(t^{b_i})=F_i(1)$ for $i\in E_\mathbf{a} \cap Z_\mathbf{b}$, 
$$\prod_{i\in A}F_i(1)\prod_{i\not\in A} F_i(t^{b_i})=\prod_{i\in E_\mathbf{a}\cap Z_\mathbf{b}}F_i(1)\prod_{i\not\in E_\mathbf{a}\cap Z_\mathbf{b}} F_i(t^{b_i}).$$
Also, since $a_i$ and $b_i$ are all even for $i\in E_\mathbf{a} \cap Z_\mathbf{b}$, $\prod_{i\not\in A}\delta(a_i,b_i)=\prod_{i\not\in E_\mathbf{a}\cap Z_\mathbf{b}} \delta(a_i,b_i)$. So the only remaining term which is depended on the choice of $A$ is $(-1)^{|A|}$. 

$$T_\mathbf{b}=2^n\cdot \prod_{i\in E_\mathbf{a}\cap Z_\mathbf{b}}F_i(1)\prod_{i\not\in E_\mathbf{a}\cap Z_\mathbf{b}} F_i(t^{b_i})\prod_{i\not\in E_\mathbf{a}\cap Z_\mathbf{b}} \delta(a_i,b_i) \sum_{Z_\mathbf{a}\subset A\subset E_\mathbf{a}\cap Z_\mathbf{b}} (-1)^{|A|}$$

\noindent\textbf{Step 5}

If $Z_\mathbf{a}\neq E_\mathbf{a}\cap Z_\mathbf{b}$, then $\sum_{Z_\mathbf{a}\subset A\subset E_\mathbf{a}\cap Z_\mathbf{b}} (-1)^{|A|}=(-1)^{|Z_\mathbf{a}|}(1-1)^{|E_\mathbf{a}\cap Z_\mathbf{b}\backslash Z_\mathbf{a}|}=0$. It means that $T_\mathbf{b}=0$ when $Z_\mathbf{a}\neq E_\mathbf{a}\cap Z_\mathbf{b}$. 

We can assume that $Z_\mathbf{a}=E_\mathbf{a}\cap Z_\mathbf{b}$. Since $\sum_{Z_\mathbf{a}\subset A\subset E_\mathbf{a}\cap Z_\mathbf{b}} (-1)^{|A|}=(-1)^{|Z_\mathbf{a}|}$,
$$T_\mathbf{b}=(-1)^{|Z_\mathbf{a}|}\cdot 2^n\cdot \prod_{i\not\in Z_\mathbf{a}} F_i(t^{b_i})\prod_{i\not\in Z_\mathbf{a}}\delta(a_i,b_i)$$

For some $i\not\in Z_\mathbf{a}$, if $a_i$ and $b_i$ do not have same parity, $\delta(a_i,b_i)=0$ and it implies that $T_\mathbf{b}=0$.

Now we can assume that $a_i\equiv b_i~(\mbox{mod} ~2)$ for all $i\not\in Z_\mathbf{a}$. Suppose that there exist an index $i$ such that $a_i\neq0$ and $b_i=0$. Since $a_i$ and $b_i$ have the same parity, $a_i$ is even. It means that $i\in E_\mathbf{a}$ and $i\in Z_\mathbf{b}$. But $i\in E_\mathbf{a}\cap Z_\mathbf{b}=Z_\mathbf{a}$ contradicts to the supposition. It means that $Z_\mathbf{a}=Z_\mathbf{b}$.

 If $a_i=1$, then $b_i\neq 0$ and so $b_i\geq a_i$. If $a_i=2$, then $b_i\neq 0$ and $b_i\equiv 0~(\mbox{mod~} 2)$. So $b_i\geq 2=a_i$. Suppose that $a_i>2$ and $b_i<a_i$ for some $i\not\in Z_\mathbf{a}$. Then $b_i=a_i-2j$ for some $1\leq j\leq m_i=\lfloor \frac{a_i-1}{2}\rfloor$ since  $a_i$ and $b_i$ have same parity and $0<b_i<a_i$. It implies that $F_i(t^{b_i})=\prod_{j=1}^{m_i}(t^{b_i}-t^{a_i-2j})=0$ for some $i\not\in Z_\mathbf{a}$. Hence
$$T_\mathbf{b}=(-1)^{|Z_\mathbf{a}|}\cdot 2^n\cdot \prod_{i\not\in Z_\mathbf{a}} F_i(t^{b_i})=0$$

So nonzero $T_\mathbf{b}$ occurs only when $b_i\geq a_i$ for all $i\not\in Z_\mathbf{a}$. Since the total degree $\sum_{i=0}^n b_i=\sum_{i=0}^n a_i=d$ are same, the only nonzero $T_\mathbf{b}$ occurs when $b_i=a_i$ for all $i$. 

In conclusion, 

$$D_\mathbf{a}=\binom{d}{\mathbf{a}}\cdot T_\mathbf{a}=\binom{d}{\mathbf{a}}\cdot (-1)^{|Z_\mathbf{a}|}\cdot 2^n \cdot \prod_{i\not\in Z_\mathbf{a} }F_i(t^{a_i})$$
\end{proof}

How many do linear forms appear in the decomposition using (\ref{eqeq0})? If $E_\mathbf{a}\neq\{0,\ldots,n\}$, the number of linear forms in Theorem \ref{coeff} is given by 
\begin{align*}
&\sum_{Z_\mathbf{a}\subset A\subset E_\mathbf{a}}|K_A||\overline{S_A}|=\sum_{Z_\mathbf{a}\subset A\subset E_\mathbf{a}}2^{n-|A|}\prod_{i\not\in A} (m_i+1)=\frac{1}{2}\sum_{Z_\mathbf{a}\subset A\subset E_\mathbf{a}}\prod_{i\not\in A} 2(m_i+1)\\
&=\frac{1}{2}\sum_{Z_\mathbf{a}\subset A\subset E_\mathbf{a}}\prod_{i\not\in E_\mathbf{a}} (a_i+1)\prod_{i\not\in A,i\in E_\mathbf{a}}(a_i)=\frac{1}{2}\prod_{i\not\in E_\mathbf{a}} (a_i+1)\sum_{Z_\mathbf{a}\subset A\subset E_\mathbf{a}}\prod_{i\not\in A,i\in E_\mathbf{a}}(a_i)\\
&=\frac{1}{2}\prod_{i\not\in E_\mathbf{a}}(a_i+1)\prod_{i\in E_\mathbf{a}\backslash Z_\mathbf{a}} (a_i+1)=\frac{1}{2}\prod_{i\not\in Z_\mathbf{a}}(a_i+1).
\end{align*}
In case of $E_\mathbf{a}=\{0,\ldots,n\}$ (i.e. all the elements $a_i$ are even), the choice $A=\{0,\ldots,n\}$ should not be counted since there is no related linear form. Hence the toal number of linear forms can be computed as $$\frac{1}{2}\sum_{Z_\mathbf{a}\subset A\not\subset \{0,\ldots,n\}}\prod_{i\not\in A} (a_i)=\frac{1}{2}(\prod_{i\not\in Z_\mathbf{a}}(a_i+1)-1)\cdots(\ast).$$ 

Since, in the argument above we do not consider the parallel shift of the entries in the set $K_A$ due to non-zero scaling of corresponding linear forms, the number $(\ast)$ is larger than the upper bound given in the paper \cite{HM}.
To remedy this, instead of $K_A$ we need to choose $k_i$'s from $$\overline{K_A}=\{(k_i)_{i\not\in A}~|~k_i\in \mathbb{Z}, 0\leq k_i\leq m_i \mathrm{~for~}  i\not\in A, \min\{k_i\}_{i\not\in A}=0\}~,$$
which is a set of representatives of each class under the shift. Then, based on Theorem \ref{coeff}, we can obtain a more optimized (i.e number of linear forms reduced as much as possible) result as follows.

\begin{corollary}\label{reduced_decomp}
Let $\mathbf{a}\in\mathbb{Z}_{\geq 0}^{n+1}$ be a sequence of $n+1$ nonnegative integers with $|\mathbf{a}|=d$. 
Then, we have
\begin{equation}\label{eqeq}
D_\mathbf{a}\cdot X_0^{a_0}X_1^{a_1}\cdots X_n^{a_n}=\sum_{Z_\mathbf{a}\subset A\subset E_\mathbf{a}}\sum_{\mathbf{k}\in \overline{K_A}}\sum_{\mathbf{s}\in \overline{S_A}}\overline{C_{A,\mathbf{k},\mathbf{s}}}(\ell_{A,\mathbf{k},\mathbf{s}})^d\end{equation}
where
$$D_\mathbf{a}=(-1)^{|Z_\mathbf{a}|}\cdot 2^n \cdot \binom{d}{\mathbf{a}}\cdot\prod_{i\not\in Z_\mathbf{a}} F_i(t^{a_i})$$ and
$$\overline{C_{A,\mathbf{k},\mathbf{s}}}=\sum_{j=0}^{\min\{m_i-k_i\}_{i\not\in A}} t^{d\cdot j}C_{A,\mathbf{k}+j\cdot\mathbf{1},\mathbf{s}}~.$$

\end{corollary}
\begin{proof}
By Theorem \ref{coeff}, 

$$D_\mathbf{a}\cdot X_0^{a_0}X_1^{a_1}\cdots X_n^{a_n}=\sum_{Z_\mathbf{a}\subset A\subset E_\mathbf{a}}\sum_{\mathbf{g}\in K_A}\sum_{\mathbf{s}\in \overline{{S}_A}}C_{A,\mathbf{g},\mathbf{s}}(\ell_{A,\mathbf{g},\mathbf{s}})^d.$$
For a set $(A,\mathbf{g},\mathbf{s})$ with $Z_\mathbf{a}\subset A\subset E_\mathbf{a}, \mathbf{g}\in K_A, \mathbf{s}\in \overline{{S}_A}$, let $m=\min\{g_i\}_{i\not\in A}$. Then
$$\ell_{A,\mathbf{g},\mathbf{s}}=\sum_{i\not\in A}(-1)^{s_i}t^{g_i}X_i=t^m\big(\sum_{i\not\in A}(-1)^{s_i}t^{g_i-m}X_i\big)=t^m\ell_{A,\mathbf{h},\mathbf{s}}$$ where $\mathbf{h}=(g_i-m)_{i\not\in A}.$ Since $\min\{h_i\}_{i\not\in A}=\min\{g_i\}-\min\{g_i\}=0$, $\mathbf{h}\in \overline{K_A}$. 
Hence
\begin{align*}
&\sum_{Z_\mathbf{a}\subset A\subset E_\mathbf{a}}\sum_{\mathbf{g}\in K_A}\sum_{\mathbf{s}\in \overline{{S}_A}}C_{A,\mathbf{g},\mathbf{s}}(\ell_{A,\mathbf{g},\mathbf{s}})^d\\
&=\sum_{Z_\mathbf{a}\subset A\subset E_\mathbf{a}}\sum_{\mathbf{h}\in \overline{K_A}}\sum_{\mathbf{s}\in \overline{{S}_A}}\sum_{\substack{\mathbf{g}\in K_A,\\\mathbf{g}-\min\{g_i\}\cdot\mathbf{1}=\mathbf{h}}}C_{A,\mathbf{g},\mathbf{s}}(t^{\min\{g_i\}}\cdot\ell_{A,\mathbf{h},\mathbf{s}})^d,
\end{align*}
and
\begin{equation}\label{summation}\sum_{\substack{\mathbf{g}\in K_A,\\\mathbf{g}-\min\{g_i\}\cdot\mathbf{1}=\mathbf{h}}}C_{A,\mathbf{g},\mathbf{s}}\cdot t^{\min\{g_i\}\cdot d}=\sum_{\mathbf{h}+m\cdot\mathbf{1}\in K_A}C_{A,\mathbf{h}+m\cdot\mathbf{1},\mathbf{s}}\cdot t^{m\cdot d}~.
\end{equation}
Since $\mathbf{h}+m\cdot\mathbf{1}\in K_A$ if and only if $0\leq h_i+m\leq m_i$ for $i\not\in A$, $$0=\max\{-h_i\}_{i\not\in A}\leq m\leq \min\{m_i-h_i\}_{i\not\in A}~.$$
Now the summation (\ref{summation}) becomes
$$\sum_{m=0}^{\min\{m_i-h_i\}_{i\not\in A}} C_{A,\mathbf{h}+m\cdot\mathbf{1},\mathbf{s}}\cdot t^{m\cdot d}=:\overline{C_{A,\mathbf{h},\mathbf{s}}}~.$$
\end{proof}

\begin{remark}[Waring rank and decomposition of a monomial over any infinite field $\kk$]\label{kk-rank} We would like to remark that the Waring type identity in Corollary \ref{reduced_decomp} gives an upper bound for the Waring rank as  
\begin{equation}\label{Our_bd_rank_over_kk}
\rank_\kk(M)\leq \frac{1}{2}(\prod_{i=0}^n(a_i+1)-\prod_{i=0}^n(a_i-1))~.
\end{equation}
for any monomial $M=X_0^{a_0}X_1^{a_1}\ldots X_n^{a_n}$ with $a_0\ge a_1\ge \cdots \ge a_n>0$ over an infinite field $\kk$ in more direct way than in \cite{HM}; i.e. just by counting linear forms in the given explicit decomposition. It's because the total number of linear forms appearing in (\ref{eqeq}) is counted by
\begin{align*}
&\sum_{Z_\mathbf{a}\subset A\subset E_\mathbf{a}}|\overline{K_A}||\overline{S_A}|=\sum_{Z_\mathbf{a}\subset A\subset E_\mathbf{a}} \bigg(\prod_{i\not\in A} (m_i+1)-\prod_{i\not\in A}(m_i)\bigg)\cdot 2^{n-|A|}\\
&=\frac{1}{2}\sum_{Z_\mathbf{a}\subset A\subset E_\mathbf{a}}\bigg(\prod_{i\not\in A} 2(m_i+1)-\prod_{i\not\in A}2(m_i)\bigg)\\
&=\frac{1}{2}\sum_{Z_\mathbf{a}\subset A\subset E_\mathbf{a}}\bigg(\prod_{i\not\in E_\mathbf{a}}(a_i+1)\prod_{i\in E_\mathbf{a}\backslash A}(a_i)-\prod_{i\not\in E_\mathbf{a}} (a_i-1)\prod_{i\in E_\mathbf{a}\backslash A} (a_i-2)\bigg)\\
&=\frac{1}{2}\prod_{i\not\in E_\mathbf{a}}(a_i+1)\sum_{Z_\mathbf{a}\subset A\subset E_\mathbf{a}}\prod_{i\in E_\mathbf{a}\backslash A}(a_i)-\frac{1}{2}\prod_{i\not\in E_\mathbf{a}}(a_i-1)\sum_{Z_\mathbf{a}\subset A\subset E_\mathbf{a}}\prod_{i\in E_\mathbf{a}\backslash A}(a_i-2)\\
&=\frac{1}{2}\prod_{i\not\in E_\mathbf{a}}(a_i+1)\prod_{i\in E_\mathbf{a}\backslash Z_\mathbf{a}}(a_i+1)-\frac{1}{2}\prod_{i\not\in E_\mathbf{a}}(a_i-1)\prod_{i\in E_\mathbf{a}\backslash Z_\mathbf{a}}(a_i-1)\\
&=\frac{1}{2}\bigg(\prod_{i\not\in Z_\mathbf{a}}(a_i+1)-\prod_{i\not\in Z_\mathbf{a}}(a_i-1)\bigg)~,
\end{align*}
which is the same number as in the upper bound (\ref{Our_bd_rank_over_kk}).


\end{remark}

\begin{remark}[Linear forms $\ell_{A,\mathbf{k},\mathbf{s}}$'s via Apolarity]
The choice for our linear forms $\ell_{A,\mathbf{k},\mathbf{s}}$'s in this paper is originated from the apolarity using the ideal $J_\mathbf{a}(t)\subset(x_0^{a_0+1},\ldots,x_n^{a_n+1})$, which is first introduced in \cite{HM}. Each point of the zero set $V(J_\mathbf{a}(t))$ in the projective $n$-space over $\kk$ decides exactly a linear form $\ell_{A,\mathbf{k},\mathbf{s}}$ in the present paper up to non-zero scaling.

\end{remark}

\begin{example}[Case of $X_0^4X_1^3X_2^2$]\label{432case}
Let $\mathbf{a}=\{4,3,2\}$. Then $Z_\mathbf{a}=\emptyset$, $E_\mathbf{a}=\{0,2\}$ and $m_0=1,m_1=1,m_2=0$. There are four cases $A\subset\{0,2\}$ where $A=\emptyset, \{0\},\{2\},\{0,2\}$ and for each $i$ we have $$F_0=(y-t^2),\quad F_1=(y-t),\quad F_2=1~.$$
\begin{enumerate}
\item $A=\emptyset$. Then, we get
\begin{align*}
K_A=\overline{K_A}&=\{(k_0,k_1,k_2)~|~0\leq k_0\leq 1,0\leq k_1\leq 1, 0\leq k_2\leq 0\}\\
&=\{(0,0,0),(0,1,0),(0,0,1),(0,1,1)\} \\
\overline{S_A}&=\{(0,0,0),(0,0,1),(0,1,0),(0,1,1)\}.
\end{align*}
The linear forms and coefficients indexed by $(\mathbf{k}, \mathbf{s})\in K_A\times \overline{S_A}$ is given by the following tables

\begin{center}
$\ell_{\emptyset,\mathbf{k},\mathbf{s}}= \begin{tabular}{|c|c|c|c|c|}\hline
& $\mathbf{s}=000$ & $001$ & $010$ & $011$\\ \hline
$\mathbf{k}=000$ &$X_0+X_1+X_2 $ &$X_0+X_1-X_2 $ &$X_0-X_1+X_2 $& $X_0-X_1-X_2 $ \\ \hline
$010$ & $X_0+tX_1+X_2 $& $X_0+tX_1-X_2 $&$X_0-tX_1+X_2 $& $X_0-tX_1-X_2 $\\ \hline
$001$ & $X_0+X_1+tX_2 $&$X_0+X_1-tX_2 $ &$X_0-X_1+tX_2 $ & $X_0-X_1-tX_2 $\\ \hline
$011$ &$X_0+tX_1+tX_2 $ & $X_0+tX_1-tX_2 $&$X_0-tX_1+tX_2 $ & $X_0-tX_1-tX_2 $\\\hline
\end{tabular}$
\end{center}

\begin{center}
$C_{\emptyset,\mathbf{k},\mathbf{s}}=(-1)^0\cdot 2^0\cdot 1\cdot$ \begin{tabular}{|c|c|c|c|c|}\hline
& $\mathbf{s}=000$ & $001$ & $010$ & $011$\\ \hline
$\mathbf{k}=000$ &$t^3$ &$t^3$ &$-t^3$& $-t^3$ \\ \hline
$010$ & $-t^2$& $-t^2$&$t^2$& $t^2$\\ \hline
$001$ & $-t$&$-t$ &$t$ & $t$\\ \hline
$011$ &$1$ & $1$&$-1$ & $-1$\\\hline
\end{tabular}
\end{center}

\item $A=\{0\}$
Then
\begin{align*}
K_A=\overline{K_A}&=\{(k_1,k_2)~|~0\leq k_1\leq 1, 0\leq k_2\leq 0\}\\
&=\{(0,0),(1,0)\} \\
\overline{S_A}&=\{(0,0),(0,1)\}.
\end{align*}
The linear forms and coefficients defined by $(\mathbf{k}, \mathbf{s})\in K_A\times \overline{S_A}$ is given by the following tables

\begin{center}
$\ell_{\{0\},\mathbf{k},\mathbf{s}}=$ 
\begin{tabular}{|c|c|c|}\hline
-& $\mathbf{s}=00$ & $01$\\ \hline
$\mathbf{k}=00$ & $X_1+X_2$& $X_1-X_2$\\ \hline
$10$ & $tX_1+X_2$&$tX_1-X_2$ \\ \hline
\end{tabular}
\end{center}

\begin{center}
$C_{\{0\},\mathbf{k},\mathbf{s}}=(-1)^1\cdot 2^1\cdot (1-t^2)\cdot$ 
\begin{tabular}{|c|c|c|}\hline
- & $\mathbf{s}=00$ & $01$\\ \hline
$\mathbf{k}=00$ &$-t$ & $-t$\\ \hline
$10$ & $1$ & $1$\\ \hline
\end{tabular}
\end{center}

\item $A=\{2\}$
Then
\begin{align*}
K_A&=\{(k_0,k_1)~|~0\leq k_0\leq 1, 0\leq k_1\leq 1\}\\
&=\{(0,0),(0,1),(1,0),(1,1)\} \\
\overline{K_A}&=\{(k_0,k_1)~|~0\leq k_0\leq 1, 0\leq k_1\leq 1, \min\{k_0,k_1\}=0\}\\
&=\{(0,0),(0,1),(1,0)\}\\
\overline{S_A}&=\{(0,0),(0,1)\}.
\end{align*}
The linear forms and coefficients defined by $(\mathbf{k}, \mathbf{s})\in K_A\times \overline{S_A}$ is given by the following tables

\begin{center}
$\ell_{\{2\},\mathbf{k},\mathbf{s}}=$ 
\begin{tabular}{|c|c|c|}\hline
-& $\mathbf{s}=00$ & $01$\\ \hline
$\mathbf{k}=00$ & $X_0+X_1$& $X_0-X_1$\\ \hline
$01$ & $X_0+tX_1$&$X_0-tX_1$ \\ \hline
$10$ & $tX_0+X_1$ & $tX_0-X_1$\\ \hline
$11$ & $tX_0+tX_1$ & $tX_0-tX_1$\\ \hline
\end{tabular}
\end{center}

\begin{center}
$C_{\{2\},\mathbf{k},\mathbf{s}}=(-1)^1\cdot 2^1\cdot 1\cdot$ 
\begin{tabular}{|c|c|c|}\hline
- & $\mathbf{s}=00$ & $01$\\ \hline
$\mathbf{k}=00$ &$t^3$ & $-t^3$\\ \hline
$01$ & $-t^2$ & $t^2$\\ \hline
$10$ & $-t$ & $t$\\ \hline
$11$ & $1$ & $-1$\\ \hline
\end{tabular}
\end{center}
Since $tX_0\pm tX_1$ and $X_0\pm X_1$ represent the same linear form, the proper coefficient $\overline{C_{\{2\},(0,0),\mathbf{s}}}$ of $X_0\pm X_1$ is $\mp 2(t^3+ t^9)$.

\item $A=\{0,2\}$
Then
\begin{align*}
K_A&=\{(k_1)~|~0\leq k_1\leq 1\}\\
&=\{(0),(1)\} \\
\overline{K_A}&=\{(k_1)~|~0\leq k_1\leq 1, \min\{k_1\}=0\}\\
&=\{(0)\}\\
\overline{S_A}&=\{(0)\}.
\end{align*}
The linear forms and coefficients defined by $(\mathbf{k}, \mathbf{s})\in K_A\times \overline{S_A}$ is given by the following tables
\begin{center}
$\ell_{\{0,2\},\mathbf{k},\mathbf{s}}=$ 
\begin{tabular}{|c|c|}\hline
-& $\mathbf{s}=0$\\ \hline
$\mathbf{k}=0$ & $X_1$\\ \hline
$1$ & $tX_1$\\ \hline
\end{tabular}
\end{center}

\begin{center}
$C_{\{0,2\},\mathbf{k},\mathbf{s}}=(-1)^2\cdot 2^2\cdot (1-t^2)\cdot 1\cdot$ 
\begin{tabular}{|c|c|}\hline
- & $\mathbf{s}=0$\\ \hline
$\mathbf{k}=0$ &$-t$\\ \hline
$1$ & $1$ \\ \hline
\end{tabular}
\end{center}
Since the linear forms $X_1$ and $tX_1$ represent the same linear form, the proper coefficient $\overline{C_{\{0,2\},(0),(0)}}$ of linear form $X_1$ is $4(1-t^2)(-t+t^9)=-4t^{11}+4t^9+4t^3-4t$. 

\end{enumerate}
Finally, we compute 
\begin{equation*}
D_{\{4,3,2\}}=2^2\cdot \binom{9}{4,3,2}\cdot (t^4-t^2)(t^3-t)=5040t^3(t^2-1)^2~.
\end{equation*}
So, the identity (\ref{eqeq}) leads to 
\begin{align*}
&5040t^3(t^2-1)^2X_0^4X_1^3X_2^2=t^3(X_0+X_1+X_2)^9+t^3(X_0+X_1-X_2)^9-t^3(X_0-X_1+X_2)^9\\
&-t^3(X_0-X_1-X_2)^9-t^2(X_0+tX_1+X_2)^9-t^2(X_0+tX_1-X_2)^9+t^2(X_0-tX_1+X_2)^9\\
&+t^2(X_0-tX_1-X_2)^9-t(X_0+X_1+tX_2)^9-t(X_0+X_1-tX_2)^9+t(X_0-X_1+tX_2)^9\\
&+t(X_0-X_1-tX_2)^9+(X_0+tX_1+tX_2)^9+(X_0+tX_1-tX_2)^9-(X_0-tX_1+tX_2)^9\\
&-(X_0-tX_1-tX_2)^9-2(1-t^2)(-t)(X_1+X_2)^9-2(1-t^2)(-t)(X_1-X_2)^9\\
&-2(1-t^2)(tX_1+X_2)^9-2(1-t^2)(tX_1-X_2)^9-2(t^3+t^9)(X_0+X_1)^2+2(t^3+t^9)(X_0-X_1)^9\\
&-2(-t^2)(X_0+tX_1)^9-2t^2(X_0-tX_1)^9-2(-t)(tX_0+X_1)^9-2t(tX_0-X_1)^9\\
&+4(1-t^2)(-t+t^9)X_1^9~
\end{align*}
and as dividing by $D_{\{4,3,2\}}=5040t^3(t^2-1)^2$ we eventually have
{\small\begin{align*}
&X_0^4X_1^3X_2^2=\frac{1}{5040}\bigg[\frac{1}{(t^2-1)^2}(X_0+X_1+X_2)^9+\frac{1}{(t^2-1)^2}(X_0+X_1-X_2)^9-\frac{1}{(t^2-1)^2}(X_0-X_1+X_2)^9\\
&-\frac{1}{(t^2-1)^2}(X_0-X_1-X_2)^9-\frac{1}{t(t^2-1)^2}(X_0+tX_1+X_2)^9-\frac{1}{t(t^2-1)^2}(X_0+tX_1-X_2)^9\\
&+\frac{1}{t(t^2-1)^2}(X_0-tX_1+X_2)^9+\frac{1}{t(t^2-1)^2}(X_0-tX_1-X_2)^9-\frac{1}{t^2(t^2-1)^2}(X_0+X_1+tX_2)^9\\
&-\frac{1}{t^2(t^2-1)^2}(X_0+X_1-tX_2)^9+\frac{1}{t^2(t^2-1)^2}(X_0-X_1+tX_2)^9+\frac{1}{t^2(t^2-1)^2}(X_0-X_1-tX_2)^9\\
&+\frac{1}{t^3(t^2-1)^2}(X_0+tX_1+tX_2)^9+\frac{1}{t^3(t^2-1)^2}(X_0+tX_1-tX_2)^9-\frac{1}{t^3(t^2-1)^2}(X_0-tX_1+tX_2)^9\\
&-\frac{1}{t^3(t^2-1)^2}(X_0-tX_1-tX_2)^9-\frac{2}{t^2(t^2-1)}(X_1+X_2)^9-\frac{2}{t^2(t^2-1)}(X_1-X_2)^9\\
&+\frac{2}{t^3(t^2-1)}(tX_1+X_2)^9+\frac{2}{t^3(t^2-1)}(tX_1-X_2)^9-\frac{2(t^6+1)}{(t^2-1)^2}(X_0+X_1)^2+\frac{2(t^6+1)}{(t^2-1)^2}(X_0-X_1)^9\\
&+\frac{2}{t(t^2-1)^2}(X_0+tX_1)^9-\frac{2}{t(t^2-1)^2}(X_0-tX_1)^9+\frac{2}{t^2(t^2-1)^2}(tX_0+X_1)^9-\frac{2}{t^2(t^2-1)^2}(tX_0-X_1)^9\\
&-\frac{4(t^2+1)(t^4+1)}{t^2}X_1^9\bigg]~,
\end{align*}}
which gives us a $1$-dimensional family of a Waring decomposition of $X_0^4X_1^3X_2^2$ (i.e. for all $t\in \kk$ with $5040 t^3(t^2-1)^2\neq0$) of $\frac{1}{2}(5\cdot4\cdot3-3\cdot2\cdot1)=27$ summands.
\end{example}

\section{Case of arbitrary homogeneous form}\label{sect_conseq}

In this section, we consider some consequence of the results in Section \ref{sect_main}. Since a homogeneous form in $\kk[X_0,\ldots,X_n]$ can be written as a $\kk$-linear combination of monomials of the same degree, by our previous decomposition for a monomial, naturally we have a family of  \textit{explicit} Waring decompositions of any homogeneous polynomial with $\kk$-coefficients. Finding such a sum of powers of linear forms representation of a given degree form is quite important in many areas of mathematics. For instance, when $\kk=\mathbb{Q}$, it is closely related to the problem of integrating a polynomial function over a rational simplex, which is fundamental for applications such as discrete optimization, finite element methods in numerical analysis, and algebraic statistics computation (see e.g. \cite[section 1]{BBDKV11} and references therein). For computational complexity of this integration problem, Waring decomposition can be used to obtain a polynomial time algorithm for evaluating integrals of polynomials of some fixed constraint (see \cite[3.3, 3.4]{BBDKV11}). Let's briefly review some aspect of their result. 

Let $\Delta$ be a $k$-dimensional rational simplex inside $\mathbb{R}^n$ and let $f\in\mathbb{Q}[X_0,\ldots,X_n]$ be a homogeneous polynomial with rational coefficients. To compute $\int_\Delta f\mathrm{d}m$, where $\mathrm{d}m$ is the integral Lebesgue measure on the affine hull $\langle \Delta\rangle$ of the simplex (see \cite[2.1]{BBDKV11} for the precise definition), we recall a useful formula due to M. Brion as follows.

\begin{proposition}(Brion)
Let $\Delta$ be the simplex that is the convex hull of $(k+1)$ affinely independent vertices $\mathbf{s}_1,\mathbf{s}_2,\ldots,\mathbf{s}_{k+1}$ in $\mathbb{R}^n$. Let $\ell$ be a linear form which is regular w.r.t. $\Delta$, i.e., $\langle \ell,\mathbf{s}_i\rangle\neq \langle \ell,\mathbf{s}_j\rangle$ for any pair $i\neq j$. Then we have the following relation
\begin{equation}\label{Brion_formula}
\int_\Delta \ell^D \mathrm{d}m = k!~ \mathrm{vol}(\Delta, \mathrm{d}m)\frac{D!}{(D+k)!}\left(\sum_{i=1}^{k+1} \frac{\langle \ell,\mathbf{s}_i\rangle^{D+k}}{\prod_{j\neq i}\langle \ell,\mathbf{s}_i-\mathbf{s}_j\rangle} \right).
\end{equation}
\end{proposition}

We note that, even when $\ell$ is not regular, there exists a similar expansion of the integral as a sum of residues (see e.g. \cite[corollary 13]{BBDKV11}). 

Thus, once we represent a polynomial by a sum of power of rational linear forms, the integration immediately follows. And bounding the number of rational linear forms in the sum and finding all its $\mathbb{Q}$-coefficients are among the main issues for the computational complexity of evaluating integrals of polynomials over a rational simplex. 

Now, let us consider the number of the summands in a given rational Waring decomposition of $f_{gen}$, a \ti{general} homogeneous polynomial of degree $D$ in $(n+1)$-variables (here, we mean a `general' form by the one having all the monomials of total degree $D$).

In \cite{BBDKV11}, the authors used the following well-known identity, which is somewhat naive from the viewpoint of Waring rank, to consider their rational decomposition of $f_{gen}$,
\begin{equation}\label{naive_decomposition}
X_0^{a_0}X_1^{a_1}\cdots X_n^{a_n}=\frac{1}{D!}\sum_{0\leq p_i\leq a_i} (-1)^{D-(p_0+\cdots +p_n)} \tbinom{a_0}{p_0}\cdots\tbinom{a_n}{p_n}(p_0 X_0+\cdots+p_n X_n)^{D}~,
\end{equation}
where $\mathbf{a}=(a_0,\dots,a_n)$ and $D=|\mathbf{a}|=a_0+\cdots+a_n$. To count the number of summands properly, one should group together proportional linear forms among the whole decomposition of $f_{gen}$. The concept of \textit{primitive} vectors (i.e. $(p_0,\ldots,p_n)\in \mathbb{Z}_{\geq 0}^n$ with $\gcd(p_0,\ldots,p_n)=1$) precisely captures this number. Let $F(n,D)$ be this number of minimal summands in the rational Waring decomposition of $f_{gen}$ using (\ref{naive_decomposition}). Then, it is shown in \cite[lemma 16]{BBDKV11} that $F(n,D)$ is equal to
\begin{equation}\label{F(n,M)}
|\{(p_0,\ldots,p_n)\in \mathbb{Z}_{\geq 0}^n, \gcd(p_0,\ldots,p_n)=1,1\leq \sum_i p_i\leq D\}|=\sum_{d=1}^D \mu(d)\cdot\left(\binom{n+1+\lfloor\frac{D}{d}\rfloor}{n+1}-1\right)~,
\end{equation}
where $\mu$ is the M\"{o}bius function. 

Now, let us estimate this number as increasing the total degree in a fixed number of variables. As getting large $D$ with a fixed $n$, by (\ref{F(n,M)}) the asymptotic behavior of $F(n,D)$ can be calculated as
\begin{equation}\label{asymp_BBDKV}
F(n,D)=\frac{\delta(n,D)}{(n+1)!}D^{n+1}+O(D^n)~,
\end{equation}
where $\delta(n,D)=\ds\sum_{d=1}^D \frac{\mu(d)}{d^{n+1}}$. Since the Dirichlet series that generates the M\"{o}bius function is the  inverse of the Riemann zeta function $\zeta(s)$, which converges for $Re(s)>1$, we see that $\delta(n,D)\rightarrow\frac{1}{\zeta(n+1)}$ as $D\rightarrow \infty$ and $F(n,D)$ asymptotically has order of $D^{n+1}$ in this setting.


On the other hand, if we regard the rational Waring decomposition of $f_{gen}$ by considering each monomial summand using our result, that is, Corollary \ref{reduced_decomp}, much less linear forms are needed to represent the polynomial asymptotically. Let $K(n,D)$ be the number of minimal summands in the rational Waring decomposition of $f_{gen}$ via (\ref{eqeq}). 

\begin{theorem}\label{Knm}
Let $n\ge1, D\ge1$ be positive integers and $K(n,D)$ be as above. Then, we have the following formula
\begin{equation}\label{Knm_formula}
K(n,D)=\sum_{r=1}^{n+1} \left\{\binom{\lfloor\frac{D-r}{2}\rfloor+r}{r}-\binom{\lfloor\frac{D-r}{2}\rfloor}{r}\right\}\cdot2^{r-1}\cdot \binom{n+1}{r}.
\end{equation}
In particular, when $n$ is fixed and $D\rightarrow \infty$, we have
\begin{equation}\label{asymp_HM}
K(n,D)=\frac{n+1}{n!}D^n+O(D^{n-1})~,
\end{equation}
which has order of $D^{n}$ asymptotically.
\end{theorem}
\begin{proof}
Let $L_{\{i_1,\ldots,i_r\}}$ be the set of all linear forms appeared in the Waring decomposition of $f_{gen}$ via (\ref{eqeq}) such that every member of $L_{\{i_1,\ldots,i_r\}}$ is of the form $\lambda_1 X_{i_1}+\lambda_2 X_{i_2}+\cdots+\lambda_r X_{i_r}$ for some nonzero $\lambda_u$'s. Then, by the proof of Corollary     \ref{reduced_decomp} we have
{\small\begin{align*}
L_{\{i_1,\ldots,i_r\}}:=\bigg\{\sum_{j=1}^r (-1)^{s_{i_j}}t^{k_{i_j}}X_{i_j}~|~&\sum_{j=1}^r a_{i_j}=D,~0\leq k_{i_j}\leq m_{i_j}=\lfloor\frac{a_{i_j}-1}{2}\rfloor\\
&, \min\{k_{i_j}\}=0,~s_{i_j}\in \{0,1\}, \mbox{~and~}s_{i_1}=0\bigg\}~.
\end{align*}}
In this set, each $(k_{i_1},k_{i_2},\ldots,k_{i_r})$ should satisfy the following inequality
$$0\leq \sum_{j=1}^r k_{i_j}\leq \sum_{j=1}^r m_{i_j}=\sum_{j=1}^r \lfloor\frac{a_{i_j}-1}{2}\rfloor~\cdots~ (\ast\ast)$$
and note that both $\{a_{i_1},a_{i_2},\ldots,a_{i_r}\}$ and $\{a_{i_1},\ldots,a_{i_p}-1,\ldots,a_{i_q}+1,\ldots, a_{i_r}\}$ have the same $r$-tuple of $k_{i_j}$'s satisfying $(\ast\ast)$ whenever $a_{i_p}$ and $a_{i_q}$ are all even. Thus, we can assume that at most one of $a_{i_j}$ is even.
 If $D\equiv r~(\mbox{mod}~2)$ and all the $a_{i_j}$ are odd, $\sum_{j=1}^r m_{i_j}=\sum_{j=1}^r \frac{a_{i_j}-1}{2}=\frac{D-r}{2}$. If $D\not\equiv r(\mbox{mod}~2)$ and all the $a_{i_j}$ are odd except one, $\sum_{j=1}^r m_{i_j}=\frac{D-r-1}{2}$. Hence, for a given $D$, we get 
{\small\begin{align*}
L_{\{i_1,\ldots,i_r\}}:=\bigg\{\sum_{j=1}^r (-1)^{s_{i_j}}t^{k_{i_j}}X_{i_j}~|~&0\leq \sum_{j=1}^r k_{i_j}\leq \lfloor\frac{D-r}{2}\rfloor, \min\{k_{i_j}\}=0, s_{i_j}\in \{0,1\}, \mbox{~and~}s_{i_1}=0\bigg\}~
\end{align*}}
so that the number of elements in $L_{\{i_1,\ldots,i_r\}}$ as follows 
\begin{align*}
|L_{\{i_1,i_2,\ldots,i_r\}}|&=\bigg|\bigg\{(k_{i_j})~|~0\leq \sum_{j=1}^r k_{i_j}\leq \lfloor\frac{D-r}{2}\rfloor, \min\{k_{i_j}\}=0\bigg\}\bigg|\cdot\bigg|\bigg\{(s_{i_j})~|~s_{i_j}\in \{0,1\}, \mbox{~and~}s_{i_1}=0\bigg\}\bigg|\\
&=\bigg\{\binom{\lfloor\frac{D-r}{2}\rfloor+r}{r}-\binom{\lfloor\frac{D-r}{2}\rfloor}{r}\bigg\}\cdot 2^{r-1}~.
\end{align*}
Finally, as multiplying $\binom{n+1}{r}$ for the possible choices for the indices $\{i_1,\ldots,i_r\}\subset \{0,\ldots,n\}$, the equation (\ref{Knm_formula}) is obtained.

As getting large $D$ with fixed $n$, the highest power of $D$ is obtained when $r=n+1$. So, we estimate as 
{\small\begin{align*}
K(n,D)&\approx\left\{\binom{\frac{D-(n+1)}{2}+n+1}{n+1}-\binom{\frac{D-(n+1)}{2}}{r}\right\}\cdot2^{n}\cdot \binom{n+1}{n+1}=\left\{\binom{\frac{D+(n+1)}{2}}{n+1}-\binom{\frac{D-(n+1)}{2}}{r}\right\}\cdot2^{n}\\
&=\left(\frac{D^{n+1}+(n+1)^2D^n+O(D^{n-1})}{2^{n+1}(n+1)!}-\frac{D^{n+1}-(n+1)^2D^n+O(D^{n-1})}{2^{n+1}(n+1)!}\right)\cdot 2^n\\
&=\frac{n+1}{n!}D^n+O(D^{n-1})~,
\end{align*}}
where the asymptotic order is $D^n$ which is better than $D^{n+1}$ in the case of $F(n,D)$.
\end{proof}

\begin{remark}\label{F&K} We make some remarks on the theorem above.
\begin{enumerate}
\item[(a)] The formula (\ref{Knm_formula}) gives a new upper bound for $\mathbb{Q}$-Waring rank of arbitrary rational homogeneous polynomial. Unfortunately, this is not better than a naive bound ${D+n\choose n}$, which comes from $\dim \mathbb{Q}[X_0,X_1,\dots,X_n]$, asymptotically. But, the latter approach does not give an explicit linear forms of the power sum decomposition and its coefficients (remember that one should execute a massive computation to find them), whereas our method does provide a completely determined(!) power sum decomposition.
\item[(b)] Note that $F(n,D)$ and $K(n,D)$ have different orders in the above asymptote by (\ref{asymp_BBDKV}) and (\ref{asymp_HM}) (see also Table \ref{compareTable2} for a significant difference between $F(n,D)$ and $K(n,D)$). 
\item[(c)] In \cite{BBT} the authors also provide a Waring decomposition of any monomial with determined coefficients over the complex number $\mathbb{C}$. Since their rank is better than the bound (\ref{main_bd}), the approach using their decomposition would be surely better than the method via (\ref{eqeq}) for decomposing any homogeneous form. But, note that in their coefficient formula a complex number does occur in most cases (!), which puts a serious limitation for applying their method in application over real or rational numbers. 
\end{enumerate}
\end{remark}



\begin{table}[htb]
\small{\begin{tabular}{ l  r r r  r r r  r r r}
\toprule
$(n,D)$		& $(2,10)$	& (2,50) 	& (2,100) 	& (3,10) 	& (3,50) 		& (3,100) & (5,30) & (5,50)& (5,100)\\
			\midrule
$F(n,D)$		&205		&18,970	&144,871	&831	    	&286,893		&4,207,287 & 1,884,921 & 31,651,125 & 1,669,982,466\\
\hline
$K(n,D)$ 		& 133	& 3,613	& 14,713	& 696	&83,416		& 666,816 & 1,305,092&16,001,276 & 502,701,736\\
\bottomrule
\end{tabular}}\\[0.25\baselineskip]
\caption{Comparison of numbers of summands in the two Waring decompositions of $f_{gen}$, a general homogeneous polynomial of degree $D$ in $n+1$ variables, based on a previously known method (\ref{naive_decomposition}) in \cite{BBDKV11} and the method in this paper (\ref{eqeq})}\label{compareTable2}
\end{table}

\section{Macaulay2 code for the decompositions}\label{sect_M2code}

Finally, we present a \textsc{Macaulay2}\cite{M2} code which computes the Waring-type polynomial identity concerning any given monomial over $\kk$ in Section \ref{sect_main} and we execute it for $X_0^4X_1^3X_2^2$, the case in Example \ref{432case}.

\begin{lstlisting}

+ M2 --no-readline --print-width 79
Macaulay2, version 1.17
with packages: ConwayPolynomials, Elimination, IntegralClosure, 
		InverseSystems,LLLBases, MinimalPrimes, 
		PrimaryDecomposition, ReesAlgebra,
		Saturation, TangentCone

i1 : --For a given sequence a, find all the pairs (A,k,s)
     Akslister=method();

i2 : Akslister(List):=a->(
         n:=#a-1;
         Akslist:={};
         E:=for i from 0 to n list if(even(a_i)) then i else continue;
         Z:=for i from 0 to n list if(a_i==0) then i else continue;
         Alist0:=subsets(toList(set(E)-set(Z)));
         Alist:=for i from 0 to #Alist0-1 list sort(Z|Alist0_i);
         for i1 from 0 to #Alist-1 do if(#(Alist_i1)!=n+1) then (
         A:=Alist_i1;
         notInA:=sort toList(set(0..n)-set(A));
         KA:=toList((for i from 0 to #notInA-1 list 0)..
         	(for i from 0 to #notInA-1 list floor((a_(notInA_i)-1)/2)));
         KAbar:=for i from 0 to #KA-1 list if(min(KA_i)==0) 
         		then KA_i else continue;
         SA:=toList((for i from 0 to #notInA-1 list 0)..
         		({0}|for i from 1 to #notInA-1 list 1));
         for i2 from 0 to #KAbar-1 do (
     	 	kk:=KAbar_i2;
	 	for i3 from 0 to #SA-1 do(
     	  		ss:=SA_i3;
     	    		Akslist=append(Akslist,{A,notInA,kk,ss});
		)
	)
	)
	else continue;
	Akslist
     );

i3 : --For a given pair (A,k,s), find a linear form l_{A,k,s}
     linform=method();

i4 : linform(Ring, List):=(R,Aks)->(
         (A,notInA,kk,I):=toSequence(Aks);
         bR:=baseRing(R);
         t0:=sub(2,R);
         if(numgens bR!=0) then t0=bR_0;
         sum for i from 0 to #notInA-1 list 
         	(-1)^(I_i)*(t0)^(kk_i)*R_(notInA_i)
         );

i5 : --For a given monomial, find all the linear forms
     linforms=method();

i6 : linforms(RingElement):=(mon)->(
         R:=ring mon;
         a:=(exponents mon)_0;
         Akslist:=Akslister(a);
         for i from 0 to #Akslist-1 list linform(R,Akslist_i)
         );

i7 : --For a given sequence a, find F_i
     Flist=method();

i8 : Flist(List,Ring,ZZ):=(a,S,ind)->(
         mind:=floor((a_ind-1)/2);
         bS=baseRing(S);
         t0:=sub(2,S);
         if(numgens bS!=0) then t0=bS_0;
         if(mind<=0) then sub(1,S) else 
         	sub(product for i from 1 to mind list
		S_0-(t0)^(a_ind-2*i),S)
         );

i9 : --For a given pair (A,k,s), find a coefficient C_{A,k,s}
     cfs=method();

i10 : cfs(List,Ring,List):=(a,R,Aks)->(
          (A,notInA,kk,ss):=toSequence(Aks);
          bR:=baseRing(R);
          S:=bR[Y];
          C1:=(-2)^(#A)*product for i from 0 to #A-1 list 
          	sub(Flist(a,S,A_i),sub(matrix{{1}},bR));
          C2:=product for i from 0 to #notInA-1 list 
          	sub(coefficient((S_0)^(kk_i),Flist(a,S,notInA_i)),bR);
          C3:=sub((-1)^(sum for i from 0 to #ss-1 list 
          	a_(notInA_i)*ss_i),bR);
          C1*C2*C3
          );

i11 : --For a given pair (A,k,s), find a principle coefficient C_{A,k,s}
      pcfs=method();

i12 : pcfs(List,Ring,List):=(a,R,Aks)->(
          (A,notInA,kk,ss):=toSequence(Aks);
          bR:=baseRing(R);
          t0:=sub(2,R);
          if(numgens bR!=0) then t0=bR_0;
          sum for j from 0 to min(for i from 0 to #notInA-1 
          	list floor((a_(notInA_i)-1)/2)-kk_i) list 
          	(t0)^(j*(sum a))*
		cfs(a,R,{A,notInA,kk+(for i from 0 to #notInA-1 list j),ss})
          );

i13 : --For a given monomial, find the coefficient D_a
      D=method();

i14 : D(RingElement):=(mon)->(
          a:=(exponents(mon))_0;
          R:=ring mon;
          bR:=baseRing(R);
          S:=bR[Y];
          t0:=sub(2,R);
          if(numgens bR!=0) then t0=bR_0;
          Z:=for i from 0 to #a-1 list if(a_i==0) then i else continue;
          sub((-1)^(#Z)*2^(#a-1)*((sum a)!/(product for i from 0 to #a-1 list (a_i)!))*
          	(product for i from 0 to #a-1 
		list sub(Flist(a,S,i),matrix{{t0^(a_i)}})),bR)
          );

i15 : --For a given monomial, find all the coefficients
      coeffs=method();

i16 : coeffs(RingElement):=(mon)->(
          R:=ring mon;
          a:=(exponents(mon))_0;
          pl:=Akslister(a);
          for i from 0 to #pl-1 list pcfs(a,R,pl_i)
          );

i17 : --Test for m=X_0^4X_1^3X_2^2
      --Ring over a fractional field of Q[t]
      T=QQ[t]

o17 = T

o17 : PolynomialRing

i18 : fT=frac T

o18 = fT

o18 : FractionField

i19 : R=fT[X_0..X_2]

o19 = R

o19 : PolynomialRing

i20 : m=X_0^4*X_1^3*X_2^2

       4 3 2
o20 = X X X
       0 1 2

o20 : R

i21 : --linearforms related to m
      ls=linforms(m);

i22 : --corresponding coefficients related to m
      cs=coeffs(m);
i23 : --check the equality of RHS and LHS of the Corollary
      rhs=sum for i from 0 to #cs-1 list cs_i*(ls_i)^((degree m)_0)

            7         5        3  4 3 2
o23 = (5040t  - 10080t  + 5040t )X X X
                                  0 1 2

o23 : R

i24 : lhs=D(m)*m

            7         5        3  4 3 2
o24 = (5040t  - 10080t  + 5040t )X X X
                                  0 1 2

o24 : R

i25 : lhs==rhs

o25 = true
\end{lstlisting}

\section*{Acknowledgments}
The first author was supported by the National Research Foundation of Korea (NRF) grant funded by the Korea government (MSIT) (No. 2021R1F1A104818611) and the second author was supported by KIAS Individual Grant (MG083101) at Korea Institute of Advanced Study (KIAS).


\begin{thebibliography}{99}

\bibitem{AH} J. Alexander and A. Hirschowitz, \emph{Polynomial interpolation in several variables}, J. Alg. Geom., 4 (1995), no. 2, 201--222.



\bibitem{BBDKV11}
Baldoni, V., Berline, N., De Loera, J., K\"{o}ppe, M., and Vergne, M., \emph{How to integrate a polynomial over a simplex}, Mathematics of Computation, 80 (no. 273), 297--325 (2011).

\bibitem{BBT} W. Buczy\'nska, J. Buczy\'nski, Z. Teitler, \emph{Waring decompositions of monomials}, J. Algebra 378, 45--57 (2013).

\bibitem{BCG} M. Boij, E. Carlini, A. V. Geramita, \emph{Monomials as sums of powers: the real binary case}, Proc. Am. Math. Soc. 139, 3039--3043 (2011).





\bibitem{CCG} E. Carlini, M. V. Catalisano, A.V. Geramita, \emph{The solution to Waring's problem for monomials and the sum of coprime monomials}, J. Algebra 370, 5--14 (2012).

\bibitem{CKOV} E. Carlini, M. Kummer, A. Oneto and E. Ventura, \emph{On the real rank of monomials}, Math. Z., 286 (2017), 571--577.

\bibitem{HL} C. J. Hillar and L.-H. Lim, \emph{Most tensor problems are NP-hard}, Journal of the ACM. 60 (6) (2013), 1--39.

\bibitem{HM} K. Han and H. Moon, \emph{A New Bound for the Waring Rank of Monomials}, SIAM J. Appl. Algebra Geom. 6 (3) (2022), 407--‌431.   

\bibitem{IK} A. Iarrobino and V. Kanev, \emph{Power sums, Gorenstein algebras, and determinantal loci}, Lect. Notes in Math. vol. 1721, Springer-Verlag, Berlin, Appendix C by Iarrobino and S.L. Kleiman, 1999.

\bibitem{M2} D. R. Grayson and M. E. Stillman, \textsc{Macaulay 2}, \emph{a software system for research in algebraic geometry}, {\tt http://www.math.uiuc.edu/Macaulay2/}.


\end{thebibliography}
\end{document}